\documentclass[10pt]{article}

\usepackage[margin=1.3in]{geometry}

\usepackage[T1]{fontenc}
\usepackage[utf8]{inputenc}
\usepackage[USenglish, UKenglish, english]{babel}
\usepackage{amssymb}
\usepackage{amsthm}
\usepackage{amsmath}
\usepackage[mathscr]{euscript}
\usepackage{enumitem}
\usepackage{graphicx}
\usepackage{MnSymbol}
 \let\mathscr\relax
\usepackage[scr]{rsfso}
\usepackage{tikz-cd}
\usepackage{tikz}
\usetikzlibrary{matrix,arrows,decorations.pathmorphing}
\usepackage{hyperref}
\usepackage{marginnote}
\usetikzlibrary{arrows.meta}

\usepackage{comment}

\usepackage[raggedright]{titlesec}

\theoremstyle{definition}
\newtheorem{defin}{Definition}[section]
\theoremstyle{definition}

\theoremstyle{plain}
\newtheorem{theo}[defin]{Theorem}
\theoremstyle{plain}
\newtheorem{prop}[defin]{Proposition}
\theoremstyle{plain}
\newtheorem{lem}[defin]{Lemma}
\theoremstyle{plain}
\newtheorem{cor}[defin]{Corollary}
\theoremstyle{definition}

\theoremstyle{definition}

\theoremstyle{definition}

\theoremstyle{plain}

\theoremstyle{definition}

\theoremstyle{definition}

\theoremstyle{definition}
\newtheorem*{defin*}{Definition}
\theoremstyle{definition}
\newtheorem*{ex*}{Example}
\theoremstyle{plain}
\newtheorem*{theo*}{Theorem}
\theoremstyle{plain}
\newtheorem*{prop*}{Proposition}
\theoremstyle{plain}
\newtheorem*{lem*}{Lemma}
\theoremstyle{plain}
\newtheorem*{cor*}{Corollary}
\theoremstyle{definition}
\newtheorem*{rmk*}{Remark}
\theoremstyle{definition}
\newtheorem*{exe*}{Exercise}

\theoremstyle{plain}

\theoremstyle{plain}
\newtheorem{theoA}{Theorem}[section]

\theoremstyle{plain}
\newtheorem{corA}[theoA]{Corollary}

\theoremstyle{plain}
\newtheorem{conjA}[theoA]{Conjecture}

\numberwithin{equation}{section}

\usepackage{libertine}

\usepackage{graphicx}
\usepackage{xcolor}
\usepackage{titlesec}
\usepackage{microtype}

\definecolor{myblue}{RGB}{0,82,155}

\titleformat{\chapter}[display]
  {\normalfont\bfseries\color{black}}
  {\centering
    {}    
    {\fontsize{50}{50}\bfseries\selectfont\thechapter}
  }
  {5pt}
  {\titlerule[1.5pt]\vskip1pt\titlerule\vskip5pt\centering\Huge\selectfont}

\setlist[enumerate]{label=(\roman*)}

\def\irr{{\rm Irr}}

\def\ker{{\rm Ker}}

\def\syl{{\rm Syl}}

\def\aut{{\rm Aut}}

\def\ind{{\rm Ind}}

\def\n{{\mathbf{N}}}

\def\c{{\mathbf{C}}}

\def\z{{\mathbf{Z}}}

\def\O{{\mathbf{O}}}

\def\H{{\mathcal{H}}}

\def\P{{\mathcal{P}}}

\def\F{{\mathbf{F}}}
\def\E{{\mathbf{E}}}

\makeatletter
\newcommand{\uset}[3][0ex]{%
  \mathrel{\mathop{#3}\limits_{
    \vbox to#1{\kern-7\ex@
    \hbox{$\scriptstyle#2$}\vss}}}}
\makeatother

\newcommand{\wt}[1]{\widetilde{#1}} 

\newcommand{\wh}[1]{\widehat{#1}}

\usepackage{parskip}

\makeatletter
\def\thm@space@setup{%
  \thm@preskip=\parskip \thm@postskip=0pt
}
\makeatother

\usepackage{sectsty}

\allsectionsfont{\centering}

\makeatletter
\def\blfootnote{\gdef\@thefnmark{}\@footnotetext}
\makeatother

\title{\huge
{\bf The McKay Conjecture and central isomorphic character triples}
\author{Damiano Rossi}
\date{}
\blfootnote{\emph{$2010$ Mathematical Subject Classification:} $20$C$15$.
\\
\emph{Key words and phrases:} McKay Conjecture, inductive McKay condition, character triples.
\\
These results were obtained during the author's PhD at the Bergische Universit\"at Wuppertal funded by the research training group \textit{GRK2240: Algebro-geometric Methods in Algebra, Arithmetic and Topology} of the DFG. This work is also supported by the EPSRC grant EP/T004592/1. The author would like to thank Britta Sp\"ath for a careful reading of an earlier version of this paper.
}
}

\begin{document}

\selectlanguage{english}

\maketitle

\renewcommand{\theconjA}{\Alph{conjA}}

\renewcommand{\thetheoA}{\Alph{theoA}}

\renewcommand{\thecorA}{\Alph{corA}}

\begin{abstract}
We refine the reduction theorem of the McKay Conjecture proved by of Isaacs, Malle and Navarro. Assuming the inductive McKay condition, we obtain a strong version of the McKay Conjecture that gives central isomorphic character triples.
\end{abstract}

\section{Introduction}

The McKay Conjecture is one of the leading problems in representation theory of finite groups. It states that, if $p$ is a prime number and $P$ is a Sylow $p$-subgroup of a finite group $G$, then
$$\left|\irr_{p'}(G)\right|=\left|\irr_{p'}(\n_G(P))\right|,$$
where for any finite group $X$ we denote by $\irr_{p'}(X)$ the set of irreducible complex characters of $G$ whose degree is not divisible by $p$. In \cite{Isa-Mal-Nav07} Isaacs, Malle and Navarro prove a reduction theorem for the McKay Conjecture and show that the conjecture holds for every finite group with respect to the prime $p$ provided that the so-called inductive McKay condition holds for every non-abelian finite simple group with respect to the prime $p$.

The inductive McKay condition requires the existence of a bijection as the one predicted by the McKay Conjecture which gives central isomorphic character triples and is compatible with the action of automorphisms. Although this condition was originally thought for quasi-simple groups, it can be stated for arbitrary finite groups.

\begin{conjA}
\label{conj:iMcK}
Let $G\unlhd A$ be finite groups, $p$ a prime and $P$ a Sylow $p$-subgroup of $G$. Then there exists an $\n_A(P)$-stable subgroup $\n_G(P)\leq M\leq G$, with $M<G$ whenever $P$ is not normal in $G$, and an $\n_A(P)$-equivariant bijection
$$\Omega:\irr_{p'}\left(G\right)\to\irr_{p'}\left(M\right)$$
such that
$$\left(A_\chi,G,\chi\right)\geq_c\left(M\n_A(P)_\chi,M,\Omega(\chi)\right),$$
for every $\chi\in\irr_{p'}(G)$.
\end{conjA}

Observe that the above statement could equivalently be stated by taking $M=\n_G(P)$.
However this additional flexibility is fundamental when proving the result for quasi-simple groups.
It's also worth noting that, by using \cite[Theorem 2.16]{Spa18}, it's no loss of generality to assume $A=G\rtimes \aut(G)$.

The reduction theorem of Isaacs, Malle and Navarro can now be stated by saying that if Conjecture \ref{conj:iMcK} holds for every universal covering group of finite non-abelian simple groups, then the McKay Conjecture holds for every finite group.

The first attempt to prove a reduction theorem for Local-Global conjectures was made in \cite{Dad97} in the context of Dade's Projective Conjecture. According to Dade's philosophy, there should exist a refinement of the conjecture that is strong enough to hold for every finite group when shown for quasi-simple groups. In the case of Dade's Projective Conjecture such a refinement should be found in the inductive form of Dade's conjecture \cite[5.8]{Dad97} (see also \cite[Conjecture 1.2]{Spa17}). The aim of this paper is to show that Conjecture \ref{conj:iMcK} provides the sought refinement in the case of the McKay Conjecture.
We recall that a group $S$ is said to be \emph{involved} in $G$ if there exists $N\unlhd K\leq G$ such that $S\simeq K/N$.

\begin{theoA}
\label{thm:Main}
Let $G$ be a finite group and $p$ a prime. Suppose that Conjecture \ref{conj:iMcK} holds at the prime $p$ for the universal covering group of every non-abelian simple group involved in $G$. Then Conjecture \ref{conj:iMcK} holds for $G$ at the prime $p$.
\end{theoA}

Notice that, by work of Malle and Sp\"ath \cite{Mal-Spa16}, Conjecture \ref{conj:iMcK} is known to hold at the prime $p=2$ for every universal covering group of finite non-abelian simple groups. For odd primes $p$, Conjecture \ref{conj:iMcK} is know for almost all quasi-simple groups except possibly in certain cases when considering groups of Lie type ${\bf D}$ and $^2{\bf D}$ (see \cite{Mal08}, \cite{Spa12}, \cite{Cab-Spa13}, \cite{Cab-Spa17I}, \cite{Cab-Spa17II}, \cite{Cab-Spa19} and \cite{Spa21}).

\begin{corA}
\label{cor:Main p=2}
Conjecture \ref{conj:iMcK} holds for the prime $2$.
\end{corA}

We mention that strong forms of the Local-Global conjectures that are compatible with isomorphisms of character triples are of great use in representation theory of finite groups. For instance, in \cite{Nav-Spa14I}  a reduction theorem for Brauer's Height Zero Conjecture has been deduced from an analogue of Theorem \ref{thm:Main} in the context of the Alperin--McKay conjecture.

The paper is structured as follows: in Section $2$ we introduce some preliminary results on character triples while in Section $3$, assuming the inductive McKay condition, we obtain good bijections for groups whose quotient over the centre is isomorphic to a direct product of non-abelian simple groups. In the final section we prove Theorem \ref{thm:Main} by inspecting the structure of a minimal counterexample. 

\section{Preliminaries on character triples}

Let $G$ be a finite group, $N\unlhd G$ and $\vartheta\in\irr(N)$. If $\vartheta$ is $G$-invariant, then $(G,N,\vartheta)$ is a character triple. We are going to use the partial order relation $\geq_c$ on character triples as defined in \cite[Definition 10.14]{Nav18} and \cite[Definition 2.7]{Spa18}. Recall that this definition requires the existence of a pair of projective representation $(\mathcal{P},\mathcal{P}')$ associated with two character triples $(G,N,\vartheta)$ and $(H,M,\varphi)$ satisfying certain properties. Whenever we want to specify the choice of the projective representations we say that $(G,N,\vartheta)\geq_c(H,M,\varphi)$ via $(\mathcal{P},\mathcal{P}')$ or that $(\mathcal{P},\mathcal{P}')$ gives $(G,N,\vartheta)\geq_c(H,M,\varphi)$. In this case we say that $(G,N,\vartheta)$ and $(H, M,\varphi)$ are central isomorphic character triples. First, we need the following version of \cite[Theorem 3.14]{Nav-Spa14I} which shows the compatibility of the order relation $\geq_c$ with the Clifford correspondence.

\begin{lem}
\label{lem:Irreducible induction}
Let $N\unlhd G$, $\wt{G}\leq G$ and $H\leq G$ such that $G=NH$, $H=\wt{H}M$ and $\c_G(N)\leq H$. Set $M:=N\cap H$, $\wt{H}:=\wt{G}\cap H$, $\wt{M}:=\wt{G}\cap M$ and $\wt{N}:=\wt{G}\cap N$. Let $\wt{\vartheta}\in\irr(\wt{N})$ and $\wt{\varphi}\in\irr(\wt{M})$ such that $\vartheta:=\wt{\vartheta}^{N}\in\irr(N)$, $\varphi:=\wt{\varphi}^{M}\in\irr(M)$ and $(\wt{G},\wt{N},\wt{\vartheta})\geq_c(\wt{H},\wt{M},\wt{\varphi})$. Assume that induction of characters gives bijections $\ind_{\wt{J}}^{J}:\irr(\wt{J}\mid \wt{\vartheta})\to \irr(J\mid \vartheta)$ and $\ind_{\wt{J}\cap H}^{J\cap H}:\irr(\wt{J}\cap H\mid \wt{\varphi})\to \irr(J\cap H\mid \varphi)$, for every $N\leq J\leq G$ where $\wt{J}:=J\cap \wt{G}$, then $(G,N,\vartheta)\geq_c(H,M,\varphi)$.
\end{lem}

\begin{proof}
Consider a pair of projective representations $(\wt{\P},\wt{\P}')$ associated to $(\wt{G},\wt{N},\wt{\vartheta})\geq_c(\wt{H},\wt{M},\wt{\varphi})$. Arguing as in the proof of \cite[Theorem 3.14]{Nav-Spa14I}, we construct the induced projective representations $\P:=(\wt{\P})^G$ of $G$ and $\P':=(\wt{\P}')^H$ of $H$ associated respectively to $\vartheta$ and $\varphi$. Then $(\P,\P')$ is associated to $(G,N,\vartheta)\geq_c(H,M,\varphi)$.
\end{proof}

Next, we recall that the strong isomorphism of character triples associated to central isomorphic character triples (see \cite[Theorem 2.2]{Spa18} and \cite[Theorem 10.13 and Problem 10.4]{Nav18}) is compatible with the order relation $\geq_c$.

\begin{prop}
\label{prop:Bijections given by projective representations for double relations}
Let $(G,N,\vartheta)\geq_c(H,M,\varphi)$ via $(\P,\P')$ and consider the associated $\n_H(J)$-equivariant bijection $\sigma_J:\irr(J\mid \vartheta)\to\irr(J\cap H\mid \varphi)$ (see \cite[Theorem 2.2]{Spa18}). Then
$$\left(\n_G(J)_\psi,J,\psi\right)\geq_c\left(\n_H(J)_\psi,J\cap H,\sigma_J(\psi)\right),$$
for every $\psi\in\irr(J\mid \vartheta)$.
\end{prop}

\begin{proof}
Let $(\P,\P')$ be associated with $(G,N,\vartheta)\geq_c(H,M,\varphi)$. By \cite[Theorem 2.2]{Spa18} or \cite[Theorem 10.13]{Nav18} there exists a projective representation $\mathcal{Q}$ of $J$ with $N\leq \ker(\mathcal{Q})$ such that $\P_J\otimes\mathcal{Q}$ and $\P'_{J\cap H}\otimes \mathcal{Q}_{J\cap H}$ afford respectively $\psi$ and $\sigma_J(\psi)$. By \cite[Theorem 5.5]{Nav18} there exists a projective representation $\mathcal{D}$ of $\n_G(J)_\psi$ such that $\mathcal{D}_J=\P_J\otimes \mathcal{Q}_J$ and, arguing as in \cite[p. 707]{Nav-Spa14I}, relying on the proof of \cite[Theorem 8.16]{Nav98} we can find a projective representation $\wh{\mathcal{Q}}$ of $\n_G(J)_\psi$ satisfying $\mathcal{D}=\P_{\n_G(J)_\psi}\otimes \wh{\mathcal{Q}}$. Set $\mathcal{D}':=\P'_{\n_H(J)_\psi}\otimes \wh{\mathcal{Q}}_{\n_H(J)_\psi}$. Then $(\mathcal{D},\mathcal{D}')$ is associated to $(\n_G(J)_\psi,J,\psi)\geq_c(\n_H(J)_\psi,J\cap H,\sigma_J(\psi))$.
\end{proof}

We also need another basic observation that follows directly from the definition of $\geq_c$.

\begin{lem}
\label{lem:Basic properties of H-triples isomorphisms}
Let $(G,N,\vartheta)\geq_c(H,M,\varphi)$. Then $(J,N,\vartheta)\geq_c(J\cap H,M,\varphi)$, for every $N\leq J\leq G$.
\end{lem}

Given a bijection between characters sets which is compatible with $\geq_c$, we now show how to obtain another bijection lying over the starting one and with similar compatibility properties. To do so we apply Lemma \ref{lem:Irreducible induction} and Proposition \ref{prop:Bijections given by projective representations for double relations}.

\begin{prop}
\label{prop:Constructing bijections over bijections}
Let $K\unlhd A$ and $A_0\leq A$ such that $A=KA_0$. For every $H\leq A$ set $H_0:=H\cap A_0$. Let $\mathcal{S}\subseteq\irr(K)$ and $\mathcal{S}_0\subseteq \irr(K_0)$ be $A_0$-stable subsets and assume there exists an $A_0$-equivariant bijection
$$\Psi:\mathcal{S}\to\mathcal{S}_0$$
such that
$$\left(A_{\vartheta},K,\vartheta\right)\geq_c\left(A_{0,\vartheta},K_0,\Psi(\vartheta)\right),$$
for every $\vartheta\in\mathcal{S}$. Then, for every $K\leq J\leq A$, there exists an $\n_{A_0}(J)$-equivariant bijection
$$\Phi:\irr(J\mid \mathcal{S})\to\irr(J_0\mid \mathcal{S}_0)$$ 
such that
$$\left(\n_A(J)_\chi,J,\chi\right)\geq_c\left(\n_{A_0}(J)_\chi,J_0,\Phi(\chi)\right),$$
for every $\chi\in\irr(K\mid\mathcal{S})$.
Moreover, if $\mathcal{S}\subseteq\irr_{p',Q}(K)$, $\mathcal{S}_0\subseteq \irr_{p',Q}(K_0)$ and $\n_A(Q)\leq A_0$ for some $Q\in\syl_p(J)$, then $\Phi$ is a $\n_A(Q,J)$-equivariant bijection
$$\Phi:\irr_{p'}(J\mid \mathcal{S})\to\irr_{p'}(J_0\mid \mathcal{S}_0).$$
\end{prop}

\begin{proof}
Consider an $\n_{A_0}(J)$-transversal $\mathbb{S}$ in $\mathcal{S}$ and define $\mathbb{S}_0:=\{\Psi(\vartheta)\mid \vartheta\in\mathbb{S}\}$. Since $\Psi$ is $A_0$-equivariant, it follows that $\mathbb{S}_0$ is an $\n_{A_0}(J)$-transversal in $\mathcal{S}_0$. For every $\vartheta\in\mathbb{S}$, with $\vartheta_0:=\Psi(\vartheta)\in\mathbb{S}_0$, we fix a pair of projective representations $(\mathcal{P}^{(\vartheta)},\mathcal{P}^{(\vartheta_0)}_0)$ giving $(A_\vartheta,K,\vartheta)\geq_c(A_{0,\vartheta},K_0,\vartheta_0)$. Now, let $\mathbb{T}$ be an $\n_{A_0}(J)$-transversal in $\irr(J\mid \mathcal{S})$ such that every character $\chi\in\mathbb{T}$ lies above a character $\vartheta\in\mathbb{S}$ (this can be done by the choice of $\mathbb{S}$). Moreover, as $A=KA_0$, we have $J=KJ_0$ and therefore every $\chi\in\mathbb{T}$ lies over a unique $\vartheta\in\mathbb{S}$ by Clifford's theorem.

For $\chi\in\mathbb{T}$ lying over $\vartheta\in\mathbb{S}$, let $\varphi\in\irr(J_\vartheta\mid \vartheta)$ be the Clifford correspondent of $\chi$ over $\vartheta$. Set $\vartheta_0:=\Psi(\vartheta)\in\mathbb{S}_0$ and consider the $\n_{A_0}(J)_\vartheta$-equivariant bijection $\sigma_{J_\vartheta}:\irr(J_\vartheta\mid \vartheta)\to\irr(J_{0,\vartheta}\mid \vartheta_0)$ induced by our choice of projective representations $(\mathcal{P}^{(\vartheta)},\mathcal{P}_0^{(\vartheta_0)})$. Let $\varphi_0:=\sigma_{J_\vartheta}(\varphi)$. Since $\Psi$ is $A_0$-equivariant, we deduce that $J_{0,\vartheta}=J_{0,\vartheta_0}$ and therefore $\Phi(\chi):=\varphi^{J_0}$ is irreducible by the Clifford correspondence. Then, we define
$$\Phi\left(\chi^x\right):=\Phi(\chi)^x,$$ 
for every $\chi\in\mathbb{T}$ and $x\in \n_{A_0}(J)$. This defines an $\n_{A_0}(J)$-equivariant bijection $\Psi:\irr(J\mid \mathcal{S})\to\irr(J_0\mid \mathcal{S}_0)$.

We now prove the statement on character triples. By hypothesis we know that
$$\left(A_{\vartheta},K,\vartheta\right)\geq_c\left(A_{0,\vartheta},K_0,\vartheta_0\right)$$
and Proposition \ref{prop:Bijections given by projective representations for double relations} implies
$$\left(A_{J,\vartheta,\varphi},J_{\vartheta},\varphi\right)\geq_c\left(A_{0,J,\vartheta,\varphi},J_{0,\vartheta},\varphi_0\right).$$
Noticing that $A_{J,\vartheta,\varphi}=A_{J,\vartheta,\chi}$ and that $A_{J,\chi}=JA_{J,\vartheta,\chi}$ it follows from Lemma \ref{lem:Irreducible induction} that
$$\left(A_{J,\chi},J,\varphi\right)\geq_c\left(A_{0,J,\chi},J_0,\chi_0\right).$$
The last part of the statement follows immediately by Clifford theory.
\end{proof}

The final result of this section allows to construct centrally ordered character triples when dealing with a situation similar to the one described in Gallagher's theorem.

\begin{prop}
\label{prop:Triples and multiplication}
Let $N\unlhd G$ and $H\leq G$ with $G=NH$ and set $M:=N\cap H$. Let $K\unlhd G$ with $K\leq M$ and consider a $G$-invariant $\zeta\in\irr_{p'}(N)$ such that $\zeta_K\in\irr(K)$. Let $\overline{G}:=G/K$, $\overline{N}:=N/K$, $\overline{H}:=H/K$ and $\overline{M}:=M/K$ and suppose that $(\overline{G},\overline{N},\overline{\chi})\geq_c(\overline{H},\overline{M},\overline{\psi})$, for some $\overline{\chi}\in\irr_{p'}(\overline{N})$ and $\overline{\psi}\in\irr_{p'}(\overline{M})$. Then
$$(G,N,\chi\zeta)\geq_c(H,M,\psi\zeta_M),$$
where $\chi\in\irr(N)$ and $\psi\in\irr(M)$ are the lifts respectively of $\overline{\chi}$ and $\overline{\psi}$.
\end{prop}

\begin{proof}
Let $(\overline{\P},\overline{\P}')$ be a pair or projective representations associated to $(\overline{G},\overline{N},\overline{\chi})\geq_c(\overline{H},\overline{M},\overline{\psi})$ and consider the corresponding lifts $\P$ and $\P'$. Let $\mathcal{Q}$ be a projective representation of $G$ associated to $\zeta$ as in \cite[Definition 5.2]{Nav18}. Then $\P\otimes \mathcal{Q}$ and $\P'\otimes\mathcal{Q}_H$ are projective representations of $G$ and $H$ associated respectively to $\chi\zeta$ and $\psi\zeta_M$. Since $\c_G(N)K/K\leq \c_{G/K}(N/K)$, we conclude from the assumption that the pair $(\P\otimes\mathcal{Q},\P'\otimes\mathcal{Q}_H)$ gives $(G,N,\chi\zeta)\geq_c(H,M,\psi\zeta_M)$.
\end{proof}

\section{The inductive condition}

Our aim in this section will be to show how to obtain good bijections for groups whose quotient over the centre is isomorphic to a direct product of (not necessarily isomorphic) non-abelian simple groups whose universal covering groups satisfy Conjecture \ref{conj:iMcK}. This is done in Corollary \ref{cor:Constructing bijections from iMcK} which will be the main result of this section. Observe that Corollary \ref{cor:Constructing bijections from iMcK} is a slight generalization of \cite[Theorem 10.25]{Nav18} and of \cite[Corollary 3.14]{Spa18}.

\begin{lem}
\label{lem:iMck for direct products}
Let $S$ be a non-abelian simple group whose universal covering group satisfies Conjecture \ref{conj:iMcK} for the prime number $p$. Consider a non-negative integer $n$ and let $\wt{X}:=X^n$ be the universal covering group of $\wt{S}:=S^n$. Let $\wt{P}$ be a Sylow $p$-subgroup of $\wt{X}$ and set $\wt{\Gamma}:=\aut(\wt{X})_{\wt{P}}$. Then, there exists a $\wt{\Gamma}$-invariant subgroup $\n_{\wt{X}}(\wt{P})\leq \wt{M}<\wt{X}$ and a $\wt{\Gamma}H$-equivariant bijection
$$\wt{\Omega}:\irr_{p'}\left(\wt{X}\right)\to\irr_{p'}\left(\wt{M}\right)$$
such that
$$\left(\wt{X}\rtimes \wt{\Gamma}_{\wt{\vartheta}},\wt{X},\wt{\vartheta}\right)\geq_c\left(\wt{M}\rtimes \wt{\Gamma}_{\wt{\vartheta}},\wt{M},\wt{\Omega}\left(\wt{\vartheta}\right)\right),$$
for every $\wt{\vartheta}\in\irr_{p'}(\wt{X})$.
\end{lem}

\begin{proof}
This is \cite[Theorem 3.12]{Spa18}.
\end{proof}

Now, proceeding as the proof of \cite[Theorem 10.25]{Nav18} we obtain the following result. Notice that this is just a version of \cite[Theorem 10.25]{Nav18} adapted to the more general case where $M$ does not need to coincide with the normaliser of a Sylow $p$-subgroup.

\begin{prop}
\label{prop:Constructing bijections from iMcK, perfect case}
Let $K\unlhd A$ be finite groups with $K=[K,K]$ and $K/\z(K)\simeq S^n$ for a non-abelian simple group $S$ whose universal covering group satisfies Conjecture \ref{conj:iMcK}. Let $P_0$ be a Sylow $p$-subgroup of $K$. Then there exists a $\n_A(P_0)$-invariant subgroup $\n_K(P_0)\leq M<K$ and a $\n_A(P_0)H$-equivariant bijection
$$\Omega:\irr_{p'}(K)\to\irr_{p'}(M)$$
such that
$$\left(A_\vartheta,K,\vartheta\right)\geq_c\left(M\n_A(P_0)_\vartheta,M,\Omega(\vartheta)\right),$$
for every $\vartheta\in\irr_{p'}(K)$.
\end{prop}

\begin{proof}
This follows from the proof of \cite[Theorem 10.25]{Spa18} by applying Lemma \ref{lem:iMck for direct products}.
\end{proof}

Finally, we consider the case where $K$ is not necessarily perfect. To do so, we have to deal with characters of central products (we refer the reader to \cite[Section 5]{Isa-Mal-Nav07} and \cite[Section 10.3]{Nav18} for the relevant notation). First, we need a lemma.

\begin{lem}
\label{lem:Central products and triple isomorphisms}
Let $(G,N,\vartheta)\geq_c(H,M,\varphi)$ and consider $C\leq \c_G(N)$. Let $\nu\in\irr(C\cap N)$ be the unique irreducible constituent of $\vartheta_{C\cap N}$ and $\varphi_{C\cap N}$. Then
$$\left(\n_G(C)_{\vartheta\boldsymbol{\cdot}\psi},N\boldsymbol{\cdot} C,\vartheta\boldsymbol{\cdot}\psi\right)\geq_c\left(\n_H(C)_{\varphi\boldsymbol{\cdot}\psi},M\boldsymbol{\cdot} C,\varphi\boldsymbol{\cdot} \psi\right),$$
for every $\psi\in\irr(C\mid \nu)$.
\end{lem}

\begin{proof}
First recall that $C\leq \c_G(N)\leq H$ and observe that $N\boldsymbol{\cdot} C$ and $M\boldsymbol{\cdot} C$ are central products. By the assumption and applying Proposition \ref{prop:Bijections given by projective representations for double relations} with $J:=N\boldsymbol{\cdot} C$ we obtain
$$\left(\n_G(C)_{\vartheta\boldsymbol{\cdot}\psi},N\boldsymbol{\cdot} C,\vartheta\boldsymbol{\cdot}\psi\right)\geq_c\left(\n_H(C)_{(\sigma_{N\boldsymbol{\cdot} C}\vartheta\boldsymbol{\cdot}\psi)},M\boldsymbol{\cdot} C,\sigma_{N\boldsymbol{\cdot} C}(\vartheta\boldsymbol{\cdot} \psi)\right).$$
To conclude, notice that \cite[Lemma 5.1]{Isa-Mal-Nav07} implies that $\sigma_{N\boldsymbol{\cdot} C}(\vartheta\boldsymbol{\cdot} \psi)=\varphi\boldsymbol{\cdot} \psi$.
\end{proof}

We are now ready to prove the main result of this section.

\begin{cor}
\label{cor:Constructing bijections from iMcK}
Let $K\unlhd A$ be finite groups such that $K/\z(K)$ is a direct product of non-abelian simple groups whose universal covering groups satisfy Conjecture \ref{conj:iMcK}. Let $P_0$ be a Sylow $p$-subgroup of $K$. Then there exists a $\n_A(P_0)$-invariant subgroup $\n_K(P_0)\leq M<K$ and a $\n_A(P_0)$-equivariant bijection
$$\Omega:\irr_{p'}(K)\to\irr_{p'}(M)$$
such that
$$\left(A_\vartheta,K,\vartheta\right)\geq_c\left(M\n_A(P_0)_\vartheta,M,\Omega(\vartheta)\right),$$
for every $\vartheta\in\irr_{p'}(K)$.
\end{cor}

\begin{proof}
By hypothesis there exist non-isomorphic non-abelian simple groups $S_1,\dots, S_\ell$ that satisfy the inductive McKay condition and non-negative integers $n_1,\dots, n_\ell$ such that $K/\z(K)\simeq S_1^{n_1}\times \dots \times S_\ell^{n_\ell}$. Consider the subgroups $\z(K)\leq K_{0,i}\leq K$ such that $K_{0,i}/\z(K)\simeq S_i^{n_i}$ and observe that $K_i:=[K_{0,i},K_{0,i}]$ is a perfect normal subgroup of $A$ with $K_i/\z(K_i)\simeq S_i^{n_i}$, for $i=1,\dots,\ell$. If $K_0:=\z(K)$, then $K=K_0\boldsymbol{\cdot} \dotso\boldsymbol{\cdot} K_\ell$ is a central product of the subgroups $K_i$ and $Z:=\cap_{i=0}^\ell K_i$ satisfies $Z=\z([K,K])=\z(K_i)$, for all $i=1,\dots,\ell$. 

Let $\vartheta\in\irr_{p'}(K)$ and consider the unique irreducible constituent $\nu\in\irr(Z)$ of $\vartheta_Z$. By \cite[Lemma 5.1]{Isa-Mal-Nav07} there exist unique characters $\vartheta_i\in\irr_{p'}(K_i\mid \nu)$ such that $\vartheta=\vartheta_0\boldsymbol{\cdot} \dotso\boldsymbol{\cdot}\vartheta_\ell$. Set $Q_i:=P_0\cap K_i\in\syl_p(K_i)$ and $A_i:=\n_A(Q_i)$. By Proposition \ref{prop:Constructing bijections from iMcK, perfect case}, for every $i=1,\dots, \ell$, there exists an $A_i$-invariant subgroup $\n_{K_i}(Q_i)\leq M_i< K_i$ and a $A_i$-equivariant bijection
$$\Omega_i:\irr_{p'}(K_i)\to\irr_{p'}(M_i)$$
such that
$$\left(A_{\vartheta_i},K_i,\vartheta_i\right)\geq_c\left(M_iA_{i,\vartheta_i},M_i,\Omega_i(\vartheta_i)\right),$$
for every $\vartheta_i\in\irr_{p'}(K_i)$. For $i=0$, set $M_0:=K_0$ and let $\Omega_0$ be the identity map on $\irr(K_0)$. Now, the subgroup $M:=M_0\boldsymbol{\cdot} \dotso\boldsymbol{\cdot}M_\ell$ is the central product of the $M_i$'s and has the required properties. Moreover, the map
\begin{align*}
\Omega:\irr_{p'}(K)&\to\irr_{p'}(M)
\\
\vartheta_0\boldsymbol{\cdot} \dotso\boldsymbol{\cdot} \vartheta_\ell &\mapsto \Omega_0(\vartheta_0)\boldsymbol{\cdot} \dotso\boldsymbol{\cdot} \Omega_\ell(\vartheta_\ell)
\end{align*}
is a well defined $\n_A(P_0)$-equivariant bijection. It remains to check the statement on character triples. To do so, we are going to prove that
\begin{multline}
\label{eq:Constructing bijections from iMcK 1}
\left(A_{\vartheta_0\boldsymbol{\cdot} \dotso\boldsymbol{\cdot}\vartheta_\ell},K_0\boldsymbol{\cdot} \dotso\boldsymbol{\cdot}K_\ell,\vartheta_0\boldsymbol{\cdot} \dotso\boldsymbol{\cdot}\vartheta_\ell\right)\geq_c
\\
\left((M_0\boldsymbol{\cdot} \dotso\boldsymbol{\cdot}M_\ell) \n_A(Q_0,\dots, Q_\ell)_{\vartheta_0\boldsymbol{\cdot} \dotso\boldsymbol{\cdot}\vartheta_\ell},M_0\boldsymbol{\cdot} \dotso\boldsymbol{\cdot}M_\ell,\Omega_0(\vartheta_0)\boldsymbol{\cdot} \dotso\boldsymbol{\cdot}\Omega_\ell(\vartheta_\ell)\right)
\end{multline}
by induction on $\ell\geq 1$. Let $\ell=1$. By the previous section we know that
$$\left(A_{\vartheta_1},K_1,\vartheta_1\right)\geq_c\left(M_1A_{1,\vartheta_1},M_1,\Omega_1(\vartheta_1)\right)$$
and applying Lemma \ref{lem:Central products and triple isomorphisms} with $C:=K_0$ we deduce
$$\left(A_{\vartheta_0\boldsymbol{\cdot}\vartheta_1},K_0\boldsymbol{\cdot} K_1,\vartheta_0\boldsymbol{\cdot} \vartheta_1\right)\geq_c\left(M_1A_{1,\vartheta_0\boldsymbol{\cdot}\vartheta_1},K_0\boldsymbol{\cdot}M_1,\vartheta_0\boldsymbol{\cdot}\Omega_1(\vartheta_1)\right),$$
here we used the fact that $A_{\vartheta_0\boldsymbol{\cdot}\vartheta_1}\leq A_{\vartheta_0}\cap A_{\vartheta_1}$. Because $K_0=M_0$, $\Omega_0(\vartheta_0)=\vartheta_0$ and $M_1A_{1,\vartheta_0\boldsymbol{\cdot}\vartheta_1}=(M_0\boldsymbol{\cdot}M_1)\n_A(Q_0,Q_1)_{\vartheta_0\boldsymbol{\cdot} \vartheta_1}$ it follows that \eqref{eq:Constructing bijections from iMcK 1} holds for $\ell=1$. Consider now $\ell>1$. The inductive hypothesis yields
\begin{multline*}
\left(A_{\vartheta_0\boldsymbol{\cdot} \dotso\boldsymbol{\cdot}\vartheta_{\ell-1}},K_0\boldsymbol{\cdot} \dotso\boldsymbol{\cdot}K_{\ell-1},\vartheta_0\boldsymbol{\cdot} \dotso\boldsymbol{\cdot}\vartheta_{\ell-1}\right)\geq_c
\\
\left((M_0\boldsymbol{\cdot} \dotso\boldsymbol{\cdot}M_{\ell-1}) \n_A(Q_0,\dots, Q_{\ell-1})_{\vartheta_0\boldsymbol{\cdot} \dotso\boldsymbol{\cdot}\vartheta_{\ell-1}},M_0\boldsymbol{\cdot} \dotso\boldsymbol{\cdot}M_{\ell-1},\Omega_0(\vartheta_0)\boldsymbol{\cdot} \dotso\boldsymbol{\cdot}\Omega_{\ell-1}(\vartheta_{\ell-1})\right).
\end{multline*}
Noticing that $M_\ell\leq K_\ell\leq \n_A(Q_0,\dots,Q_{\ell-1})_{\vartheta_0\boldsymbol{\cdot} \dotso\boldsymbol{\cdot}\vartheta_\ell}$ and applying Lemma \ref{lem:Central products and triple isomorphisms} we deduce
\begin{multline}
\label{eq:Constructing bijections from iMcK 2}
\left(A_{\vartheta_0\boldsymbol{\cdot} \dotso\boldsymbol{\cdot}\vartheta_{\ell}},K_0\boldsymbol{\cdot} \dotso\boldsymbol{\cdot}K_{\ell},\vartheta_0\boldsymbol{\cdot} \dotso\boldsymbol{\cdot}\vartheta_{\ell}\right)\geq_c
\\
\left((M_0\boldsymbol{\cdot} \dotso\boldsymbol{\cdot}M_{\ell}) \n_A(Q_0,\dots, Q_{\ell-1})_{\vartheta_0\boldsymbol{\cdot} \dotso\boldsymbol{\cdot}\vartheta_{\ell}}, M_0\boldsymbol{\cdot} \dotso\boldsymbol{\cdot}M_{\ell-1}\boldsymbol{\cdot}K_\ell,
\right.\\ \left.
\Omega_0(\vartheta_0)\boldsymbol{\cdot} \dotso\boldsymbol{\cdot}\Omega_{\ell-1}(\vartheta_{\ell-1})\boldsymbol{\cdot}\vartheta_\ell\right).
\end{multline}
On the other hand the fact that
$$\left(A_{\vartheta_\ell},K_\ell,\vartheta_\ell\right)\geq_c\left(M_\ell A_{\ell,\vartheta_\ell},M_\ell,\Omega_\ell(\vartheta_\ell)\right)$$
together with Lemma \ref{lem:Basic properties of H-triples isomorphisms} (ii) implies
\begin{multline*}
\left((M_0\boldsymbol{\cdot} \dotso\boldsymbol{\cdot}M_{\ell}) \n_A(Q_0,\dots, Q_{\ell-1})_{\vartheta_0\boldsymbol{\cdot} \dotso\boldsymbol{\cdot}\vartheta_{\ell}},K_{\ell},\vartheta_{\ell}\right)\geq_c
\\
\left((M_0\boldsymbol{\cdot} \dotso\boldsymbol{\cdot}M_{\ell}) \n_A(Q_0,\dots, Q_{\ell})_{\vartheta_0\boldsymbol{\cdot} \dotso\boldsymbol{\cdot}\vartheta_{\ell}}, M_\ell,\Omega_{\ell}(\vartheta_{\ell})\right).
\end{multline*}
We now apply Lemma \ref{lem:Central products and triple isomorphisms} with $C:=M_0\boldsymbol{\cdot} \dotso\boldsymbol{\cdot}M_{\ell-1}$ and $\psi=\Omega_0(\vartheta_0)\boldsymbol{\cdot} \dotso\boldsymbol{\cdot}\Omega_{\ell-1}(\vartheta_{\ell-1})$ to obtain
\begin{multline}
\label{eq:Constructing bijections from iMcK 3}
\left((M_0\boldsymbol{\cdot} \dotso\boldsymbol{\cdot}M_{\ell}) \n_A(Q_0,\dots, Q_{\ell-1})_{\vartheta_0\boldsymbol{\cdot} \dotso\boldsymbol{\cdot}\vartheta_{\ell}},M_0\boldsymbol{\cdot} \dotso\boldsymbol{\cdot}M_{\ell-1}\boldsymbol{\cdot}K_{\ell},
\right.
\\
\left.\Omega_0(\vartheta_0)\boldsymbol{\cdot} \dotso\boldsymbol{\cdot}\Omega_{\ell-1}(\vartheta_{\ell-1})\boldsymbol{\cdot}\vartheta_{\ell}\right)\geq_c
\\
\left((M_0\boldsymbol{\cdot} \dotso\boldsymbol{\cdot}M_{\ell}) \n_A(Q_0,\dots, Q_{\ell})_{\vartheta_0\boldsymbol{\cdot} \dotso\boldsymbol{\cdot}\vartheta_{\ell}}, M_0\boldsymbol{\cdot} \dotso\boldsymbol{\cdot}M_{\ell},\Omega_0(\vartheta_0)\boldsymbol{\cdot} \dotso\boldsymbol{\cdot}\Omega_{\ell}(\vartheta_{\ell})\right).
\end{multline}
Now \eqref{eq:Constructing bijections from iMcK 1} follows from \eqref{eq:Constructing bijections from iMcK 2} and \eqref{eq:Constructing bijections from iMcK 3}.
\end{proof}

\section{The reduction}

In this final section we prove Theorem \ref{thm:Main}. To do so, proceeding as in \cite[Section 7]{Nav-Spa14I}, we analyse the structure of a minimal counterexample to Theorem \ref{thm:Main}.

\begin{lem}
\label{lem:Reduction 1}
Let $G\unlhd A$ be a minimal counterexample to Theorem \ref{thm:Main} with respect to $|G:\z(G)|$. Let $K\unlhd A$, $K\leq G$ such that $|G:K|<|G:\z(G)|$ and consider an $A$-invariant $\zeta\in\irr_{p'}(K)$. Then there exists a $\n_A(P)$-equivariant bijection
$$\Upsilon_\zeta:\irr_{p'}\left(G\enspace\middle|\enspace\zeta\right)\to \irr_{p'}\left(K\n_G(P)\enspace\middle|\enspace \zeta\right)$$
such that
$$\left(A_\tau,G,\tau\right)\geq_c\left(K\n_A(P)_\tau,K\n_G(P),\Upsilon_\zeta(\tau)\right),$$
for every $\tau\in\irr_{p'}(G\mid \zeta)$.
\end{lem}

\begin{proof}
We proceed as in the proof of \cite[Lemma 7.3]{Nav-Spa14I} and we apply \cite[Lemma 2.17]{Spa18} and Proposition \ref{prop:Triples and multiplication} in place respectively of \cite[Theorem 4.5]{Nav-Spa14I} and \cite[Theorem 4.6]{Nav-Spa14I}.
\end{proof}

\begin{prop}
\label{prop:Reduction 1}
Let $G\unlhd A$ be a minimal counterexample to Theorem \ref{thm:Main} with respect to $|G:\z(G)|$. Let $K\unlhd A$, $K\leq G$ such that $|G:K|<|G:\z(G)|$. Then there exists a $\n_A(P)$-equivariant bijection
$$\Upsilon_K:\irr_{p'}(G)\to \irr_{p'}(K\n_G(P))$$
such that
$$\left(A_\tau,G,\tau\right)\geq_c\left(K\n_A(P)_\tau,K\n_G(P),\Upsilon_K(\tau)\right),$$
for every $\tau\in\irr_{p'}(G)$.
\end{prop}

\begin{proof}
This follows from the proof of \cite[Proposition 7.4]{Nav-Spa14I} by replacing \cite[Theorem 3.14]{Nav-Spa14I} and \cite[Lemma 7.3]{Nav-Spa14I} respectively with Lemma \ref{lem:Irreducible induction} and Lemma \ref{lem:Reduction 1}.
\end{proof}

As an immediate consequence we obtain that, for a minimal counterexample $G$, we have $G=K\n_G(P)$ for any $K\unlhd A$ with $K\leq G$ and $|G:K|<|G:\z(G)|$.

\begin{cor}
\label{cor:Reduction 1}
Let $G\unlhd A$ be a minimal counterexample to Theorem \ref{thm:Main} with respect to $|G:\z(G)|$. Let $K\unlhd A$, $K\leq G$ such that $|G:K|<|G:\z(G)|$. Then $G=K\n_G(P)$.
\end{cor}

\begin{proof}
Suppose that $G_1:=K\n_G(P)$ is a proper subgroup of $G$. Then every non-abelian simple group involved in $G$ is also involved in $G_1$ and $|G_1:\z(G_1)|<|G:\z(G)|$. Set $A_1:=K\n_A(P)$. By the minimality of $G$ we can find a $\n_{A_1}(P)$-equivariant bijection
$$\Omega_1:\irr_{p'}(G_1)\to\irr_{p'}(\n_{G_1}(P))$$
such that
$$\left(A_{1,\vartheta},G_1,\vartheta\right)\geq_c\left(\n_{A_1}(P)_\vartheta,\n_{G_1}(P),\Omega_1(\vartheta)\right),$$
for every $\vartheta\in\irr_{p'}(G_1)$. Notice that $\n_{G_1}(P)=\n_G(P)$ and that $\n_{A_1}(P)=\n_A(P)$. Then, applying Proposition \ref{prop:Reduction 1} and composing the obtained bijection with $\Omega_1$ we obtain a contradiction. This proves that $G=K\n_G(P)$.
\end{proof}

Next we want to show that, if $G$ is a minimal counterexample and $K\unlhd A$ with $K\leq G$ such that $K$ has a Sylow $p$-subgroup which is central in $G$, then $K\leq \z(G)$. To do so we use the following result.

\begin{theo}
\label{thm:DGN and triples}
Let $A$ be a finite group and $M,K\unlhd A$ such that $K\leq M$ and $M/K$ is a $p$-group. Let $P$ be a $p$-subgroup of $M$ such that $M=KP$ and $P_0:=P\cap K\leq \z(M)$. Then there exists a $\n_A(P)$-equivariant bijection 
$$\Lambda_P:\irr_{p'}(M)\to\irr_{p'}(\n_M(P))$$
such that
$$\left(A_\vartheta,M,\vartheta\right)\geq_c\left(\n_A(P)_\vartheta,\n_M(P),\Lambda_P(\vartheta)\right),$$
for every $\vartheta\in\irr_{p'}(M)$.
\end{theo}

\begin{proof}
This follows directly from \cite[Corollary 5.14]{Nav-Spa14I}. 
\end{proof}

\begin{prop}
\label{prop:Reduction 2}
Let $G\unlhd A$ be a minimal counterexample to Theorem \ref{thm:Main} with respect to $|G:\z(G)|$. Let $K\unlhd A$, $K\leq G$ and suppose that $P_0:=P\cap K\leq\z(G)$. Then $K\leq \z(G)$.
\end{prop}

\begin{proof}
For the sake of contradiction assume $K\nleq \z(G)$. Then $|G:K\z(G)|<|G:\z(G)|$ and Corollary \ref{cor:Reduction 1} implies $G=K\z(G)\n_G(P)=K\n_G(P)$. Recall that $A=G\n_A(P)$ by the Frattini argument. Then $A=K\n_A(P)$ and the subgroup $M:=KP$ is normal in $A$ and satisfies $P_0:=K\cap P\leq \z(M)$ by hypothesis. Now, Theorem \ref{thm:DGN and triples} yields a $\n_A(P)$-equivariant bijection
$$\Lambda_P:\irr_{p'}(M)\to\irr_{p'}(\n_M(P))$$
such that
$$\left(A_\vartheta,M,\vartheta\right)\geq_c\left(\n_A(P)_\vartheta,\n_M(P),\Lambda_P(\vartheta)\right),$$
for every $\vartheta\in\irr_{p'}(M)$. Finally, after noticing that $\irr_{p'}(G)\subseteq \irr(G\mid \irr_{p',P}(M))$ and that $\irr_{p'}(\n_G(P))\subseteq \irr(\n_G(P)\mid \irr_{p',P}(\n_M(P)))$, we obtain a contradiction by applying Proposition \ref{prop:Constructing bijections over bijections} .
\end{proof}

We are finally ready to prove Theorem \ref{thm:Main}.

\begin{proof}[Proof of Theorem \ref{thm:Main}] Suppose, for the sake of a contradiction, that the result fails to hold and consider a counterexample $G\unlhd A$ minimal with respect to $|G:\z(G)|$. By Corollary \ref{cor:Reduction 1} (applied with $K:=\z(G)\O_p(G)$) it follows that $\O_p(G)\leq \z(G)$. Furthermore, as $\O_{p'}(G)\z(G)\cap D\leq \z(G)$, Proposition \ref{prop:Reduction 2} yields $\O_{p'}(G)\leq \z(G)$. As a consequence $\z(G)=\F(G)<\F^*(G)$, where $\F^*(G)=\F(G)\E(G)$ is the generalized Fitting subgroup of $G$ which is the product of the Fitting subgroup $\F(G)$ and the layer $K:=\E(G)$. Set $P_0:=P\cap K$ and observe that, replacing $P$ with a $G$-conjugate, we may assume that $P_0$ is a Sylow $p$-subgroup of $K$. Since $K\nleq \z(G)$, Proposition \ref{prop:Reduction 2} shows that $P_0:=P\cap K\nleq \z(G)$ and, as $\z(K)\leq \F(G)=\z(G)$, we obtain $P_0\nleq \z(K)$. Now, we can apply Corollary \ref{cor:Constructing bijections from iMcK} and find an $\n_A(P_0)$-invariant subgroup $\n_K(P_0)\leq M<K$ and a $\n_A(P_0)$-equivariant bijection
$$\Omega:\irr_{p'}(K)\to\irr_{p'}(M)$$
such that
$$\left(A_\vartheta,K,\vartheta\right)\geq_c\left(\n_A(P_0)_\vartheta,M,\Omega(\vartheta)\right),$$
for every $\vartheta\in\irr_{p'}(K)$.
Next, observe that $\n_A(P)\leq \n_A(P_0)$, that $G=K\n_G(P_0)$ by Corollary \ref{cor:Reduction 1} and that $\irr_{p'}(G)\subseteq \irr(G\mid \irr_{p',P}(K))$ and $\irr_{p'}(M\n_G(P_0))\subseteq\irr(M\n_G(P_0)\mid \irr_{p',P}(M))$. Applying Proposition \ref{prop:Constructing bijections over bijections} we obtain a $\n_A(P)$-equivariant bijection
$$\Pi_P:\irr_{p'}(G)\to\irr_{p'}(M\n_G(P_0))$$
such that
$$\left(A_\tau,G,\tau\right)\geq_c\left(M\n_A(P_0)_\vartheta,M\n_G(P_0),\Pi_P(\tau)\right)_{\H},$$
for every $\tau\in\irr_{p'}(G)$. Finally, by the minimality of $G$, it follows that Theorem \ref{thm:Main} holds for $M\n_G(P_0)$ (recall that $M<K$). This fact gives us a bijection $\irr_{p'}(M\n_G(P_0))\to \irr_{p'}(\n_G(P))$ which composed with $\Pi_P$ leads to the final contradiction.
\end{proof}

\bibliographystyle{alpha}
\bibliography{References}

\vspace{1cm}

DEPARTMENT OF MATHEMATICS, CITY, UNIVERSITY OF LONDON, EC$1$V $0$HB, UNITED KINGDOM.
\textit{Email address:} \href{mailto:damiano.rossi@city.ac.uk}{damiano.rossi@city.ac.uk}

\end{document}